\documentclass{lmss}
\usepackage{amsmath}
\usepackage{amsbsy}
\usepackage{amssymb}
\usepackage[all]{xypic}
\usepackage{graphicx}

\newtheorem{thm}{Theorem}[section] 
\newtheorem{lem}[thm]{Lemma}     
\newtheorem{cor}[thm]{Corollary}
\newtheorem{prop}[thm]{Proposition}

\newnumbered{assertion}{Assertion}    
\newnumbered{conjecture}{Conjecture}  
\newnumbered{defn}{Definition}
\newnumbered{hypothesis}{Hypothesis}
\newnumbered{rem}{Remark}
\newnumbered{note}{Note}
\newnumbered{observation}{Observation}
\newnumbered{problem}{Problem}
\newnumbered{question}{Question}
\newnumbered{algorithm}{Algorithm}
\newnumbered{exm}{Example}
\newunnumbered{notation}{Notation} 

\simpleequations

\title[Cluster algebras arising from cluster tubes]
 {Cluster algebras arising from cluster tubes} 

\author{Yu Zhou and Bin Zhu}

\dedication{Dedicated to the memory of Dieter Happel}

\classno{16G20, 18E30}

\extraline{Supported by the NSF of China (Grants 11131001).}

\def\s{\stackrel}

\providecommand{\add}{\mathop{\rm add}\nolimits}%
\providecommand{\Hom}{\mathop{\rm Hom}\nolimits}%
\providecommand{\End}{\mathop{\rm End}\nolimits}%
\providecommand{\Ext}{\mathop{\rm Ext}\nolimits}%
\providecommand{\Gr}{\mathop{\rm Gr}\nolimits}%
\providecommand{\coker}{\mathop{\rm coker}\nolimits}%
\providecommand{\id}{\mathop{\rm id}\nolimits}%
\providecommand{\ind}{\mathop{\rm ind}\nolimits}%

\def\text{\mbox}
\def\m{\multiput}

\def\A{\mathcal{A}}
\def\C{\mathcal{C}}
\def\D{\mathcal{D}}

\def\T{\mathcal{T}}

\def\S{\mathcal{S}}

\def\X{\mathfrak{X}}
\def\Y{\mathfrak{Y}}

\begin{document}
\maketitle

\begin{abstract}
We study the cluster algebras arising from cluster
tubes with rank bigger than $1$. Cluster tubes are $2-$Calabi-Yau triangulated categories which
contain no cluster tilting objects, but maximal rigid objects. Fix a
certain maximal rigid object $T$ in the cluster tube $\C_n$ of rank $n$. For
any indecomposable rigid object $M$ in $\C_n$, we define an analogous
$X_M$ of Caldero-Chapton's formula (or Palu's cluster character
formula) by using the geometric information of $M$. We show that
$X_M, X_{M'}$ satisfy the mutation formula when $M,M'$ form an
exchange pair, and that $X_{?}: M\mapsto X_M$ gives a bijection from
the set of indecomposable rigid objects in $\C_n$ to the set of
cluster variables of cluster algebra of type $C_{n-1}$, which
induces a bijection between the set of basic maximal rigid objects
in $\C_n$ and the set of clusters. This strengths a surprising
result proved recently by Buan-Marsh-Vatne that the combinatorics of
maximal rigid objects in the cluster tube $\C_n$ encode the
combinatorics of the cluster algebra of type $B_{n-1}$ since the
combinatorics of cluster algebras of type $B_{n-1}$ or of type
$C_{n-1}$ are the same by a result of Fomin and Zelevinsky.
As a consequence, we give a categorification of cluster
algebras of type $C$.
\end{abstract}


\section{Introduction} 
\label{Sec:intro}

Cluster algebras were introduced around 2000 by Fomin-Zelevinsky
\cite{FZ1} in order to give an algebraic and combinatorial framework for
the canonical basis of quantum groups and for the notion of total
positivity for semisimple algebraic groups, see \cite{F,FZ4} for a nice
survey on this topic and its background. Since they were introduced,
interesting connections between such algebras and several branches
of mathematics have emerged. In the categorification theory of
cluster algebras, cluster categories \cite{Ke1,CCS,BMRRT,Du1,Am}, and (stable)
module categories over preprojective algebras \cite{GLS1,GLS2,BIRS}
play a central role. They all have cluster tilting objects, which model the clusters of the
corresponding cluster algebras via Caldero-Chapoton's formula \cite{CC}
in the case of cluster categories or Gei\ss-Leclerc-Shr\"{o}er's map
\cite{GLS1} in the case of preprojective algebras. This motivates the
study of arbitrary $2-$Calabi-Yau triangulated categories with
cluster tilting objects (subcategories). Palu defined a cluster
character for any Hom-finite $2-$Calabi-Yau triangulated categories which have cluster
tilting objects \cite{Pa1,Pa2} (see also \cite{FK}). Recently Plamondon defined
cluster characters for Hom-infinite $2-$Calabi-Yau triangulated categories with
some hypotheses \cite{Pla1,Pla2}.

It was proved in \cite{IY,BIRS} that
one can mutate a cluster tilting object $T=\oplus_{i=1}^{n}T_i$ (resp. maximal
rigid object) at any indecomposable direct summand $T_i$ to get a new
cluster tilting object $\mu_i(T)$ (resp. maximal rigid object) via
exchange triangles in a $2-$Calabi-Yau triangulated category $\C$. To any
maximal rigid object $T$, one can associate an integer matrix $A_T$
by using the exchange triangles (where $A_T$ is the
transpose of $B_T$ defined in \cite{BMV}, see Section \ref{Sec:Pre}). If $A_T$ and $A_{\mu_i(T)}$
are related by Fomin-Zelevinsky's matrix mutation for any maximal
rigid object $T$ and any direct summand $T_i$, then we say that the maximal rigid objects form a
cluster structure in $\C$ \cite{BIRS,BMV}. In \cite{BMV}, Buan-Marsh-Vatne
showed that maximal rigid objects without $2-$cycles form a cluster
structure in $\C$ (see also \cite{BIRS}). Therefore cluster tilting objects and
maximal rigid objects are important objects in $2-$Calabi-Yau triangulated
categories. They have many nice properties, see for example \cite{KR,KZ,IY,DK,V,Y,ZZ}.
Cluster tilting objects are obviously
maximal rigid, but the converse is not true, see \cite{BIKR} for the
first examples. Cluster tubes provide the second examples, in which
the quivers of endomorphism algebras of
maximal rigid objects contain loops, but no $2-$cycles \cite{BMV}.
A cluster tube of rank $n$, denoted by $\C_n$, is by definition, the orbit
category by $\tau^{-1}[1]$ of the derived category of the hereditary
abelian category of nilpotent representations of the quiver with
underlying graph $\widetilde{A}_{n-1}$ and with cyclic orientation.
It is a $2-$Calabi-Yau triangulated category \cite{BKL,Ke1}. In \cite{BMV}, a
classification of maximal rigid objects in the cluster tube $\C_n$
is given. The maximal rigid objects are proved to form a cluster
structure. Furthermore they use the geometric description of the
exchange graph of the cluster algebra of type $B_{n-1}$ in \cite{FZ2} to
prove that there is a bijection between the set of indecomposable
rigid objects in the cluster tube $\C _n$ and the set of cluster
variables of the cluster algebra of type $B_{n-1}$. Under this
bijection, maximal rigid objects go to clusters. Since the
cluster combinatorics of the cluster algebra of type $C_{n-1}$ is the
same as that of the cluster algebra of type $B_{n-1}$ by Proposition
3.15 in \cite{FZ2}, there is a bijection  between the set of
indecomposable rigid objects in the cluster tube $\C _n$ and the set of
cluster variables of the cluster algebra of type $C_{n-1}$.

The aim of the paper is to study the cluster algebras arising from
cluster tubes. This is the first attempt to the well-known question
how to define cluster characters with respect to a maximal rigid
object in a $2-$Calabi-Yau triangulated category, in which maximal
rigid objects may have loops (compare to \cite{Pla1,Pla2}). We give an
analogue of Caldero-Chapoton's formula \cite{CC} (or Palu's character
\cite{Pa1}) for cluster tubes. Fix a certain basic maximal rigid object $T$ in
the cluster tube $\C_n (n>1)$. $A_T$ denotes the skew-symmetrizable
matrix associated with $T$, which is of type $C_{n-1}$ \cite{BMV} (please
see the precise meaning in Section \ref{Sec:Pre}). For any indecomposable rigid
object $M$ in $\C_n$, with respect to $T$, we define a Laurent polynomial
$X_M$. We prove that the formula $X_M$ satisfies the mutation formula for cluster variables: i.e. if $M$ and
$M^*$ are indecomposable rigid objects such that $M\oplus N$ and
$M^*\oplus N$, for some rigid object $N$, are maximal rigid objects
in $\C_n$, then $X_M\cdot X_{M^*}=X_E+X_E'$ where $E, E'$ are the
middles of the exchange triangles: $M\rightarrow E\rightarrow
M^*\rightarrow M[1]$, $M^*\rightarrow E'\rightarrow M\rightarrow
M^*[1]$. We note here that the dimension of Ext$^1(M,M^*)$ can be
$2$ (compare to the cases considered before in \cite{CC,FK,Pa1,Pla1,Pla2},
where the $k-$dimension of Ext$^1(M,M^*)$ is always one).
Thus the map $X_{?}$ gives a bijection from the set of
indecomposable rigid objects in $\C_n$ to the set of cluster
variables of the cluster algebra of type $C_{n-1}$. This gives an
explicit bijection parallel with that given by Buan-Marsh-Vatne
\cite{BMV} for type $B_{n-1}$ (since there is a natural bijection between
type $B_{n-1}$ and type $C_{n-1}$, see \cite{FZ2}). The algebra generated
by the $X_M$, where $M$ runs over all indecomposable rigid objects in
$\C_n$ is isomorphic to the cluster algebras of type $C_{n-1}$.  In
\cite{Du2}, Dupont proved the multiplication formula for cluster
characters associated to regular modules over the path algebra of
any representation-infinite quiver; Ding-Xu \cite{DX} also defined an
analogous map for cluster tubes and gave multiplication formulas.
But their formulas are not the exchange formula for cluster
variables on the one hand, and their maps can not be used to realize
the cluster structure of $\C_n$ on the other hand.

The paper is organized as follows: In Section \ref{Sec:Pre}, we recall some
basics on cluster algebras and $2-$Calabi-Yau triangulated categories. In
particular, we recall the definition of cluster tubes and basic
descriptions on indecomposable rigid objects in cluster tubes from \cite{BMV}.
In Section \ref{Sec:cha}, for any positive number $n>1$, fix a basic maximal
rigid object $T$ in $\C_n$, we calculate the index of any
indecomposable rigid object $M$ with respect to $T$ (defined in \cite{Pa1,DK,Pla1,Pla2})
and define the analogue $X_M$ of the CC-map or Palu's map for an
indecomposable rigid object $M$ with respect to $T$. This map
$X_{?}$ is called cluster map. Using the structure of the
cluster tube $\C_n$ of rank $n$, we divide the set of indecomposable
rigid objects into three disjoint subsets. Using the structure of
endomorphism algebras of $T$ \cite{BMV,V,Y,ZZ}, we calculate the
explicit formula of $X_M$ according to which subset $M$ belongs to.
In Section \ref{Sec:mu}, we prove that $X_M, X_{M^*}$ satisfy the mutation
formula when $M, M^*$ form a mutation pair. Using the mutation
triangles, we explain that the matrix $A_T$ associated with $T$ is a
skew-symmetrizable matrix of type $C_{n-1}$. We prove that the map
$X_{?}$ gives a bijection between the set of indecomposable rigid
objects in $\C_n$ and the set of cluster variables of $A_T$, which
induces a bijection between the set of basic maximal rigid
objects and the set of clusters of $A_T$. It follows that the
cluster algebra generated by $X_M$, where $M$ runs over all
indecomposable rigid objects, is isomorphic to the cluster algebra
of type $C_{n-1}$. In the final section, we give an application of
the cluster map $X_{?}$. We prove that the simplicial complex
generated by the indecomposable rigid objects in $\C_n$ gives a
realization of the cluster complex of the root system of type
$C_{n-1}$ defined in \cite{FZ2}.

\section{Preliminaries}
\label{Sec:Pre}

We recall some basic notions on cluster algebras which can be found
in the papers by Fomin and Zelevinsky \cite{FZ1,FZ2,FZ3}. The cluster algebras
we deal with in this paper are without coefficients.

Let $\mathcal{F}=\mathbb{Q}(x_1,x_2,\cdots, x_n)$ be the field of
rational functions in indeterminates $x_1,x_2,\cdots, x_n$. Set
$\underline{x}=\{x_1,x_2,\cdots, x_n\}.$ Let $A=(a_{ij})$ be an
$n\times n$ skew-symmetrizable integer matrix. For any $k\in \{1, 2, \cdots, n\}$, the
mutation $\mu_k(A)$ of $A$ in direction $k$ is by definition, an
integer matrix $A'=(a'_{ij})$, where
\[ a'_{ij}=\left\{ \begin{array}{lccccl}-a_{ij}&&&&&\mbox{if } i=k
  \mbox{ or }j=k,\\
  a_{ij}+\frac{|a_{ik}|a_{kj}+a_{ik}|a_{kj}|}{2}&&&&&
  \mbox{otherwise.}\end{array}\right. \]
$A'$ is a skew-symmetrizable matrix too. A seed is a pair $(\underline{u}, A)$, where $\underline{u}=\{u_1,u_2,\cdots,
u_n\}$ is a transcendence base of $\mathcal{F}$ and $A$ is an
$n\times n$ skew-symmetrizable integer matrix.
A mutation $\mu_k(\underline{u}, A)$ of a seed $(\underline{u}, A)$
in direction $k$ is a new seed $(\underline{u'}, A')$, where
$A'=\mu_k(A)$, and $\underline{u'}=(\underline{u}\setminus\{u_k\})\bigcup
\{u'_k\}$, where $u'_k$ is defined in the following mutation formula:
\[u_ku'_k=\prod_{a_{ik}>0}u_i^{a_{ik}}+\prod_{a_{ik}<0}u_i^{-a_{ik}}.\]

The cluster algebra $\mathcal{A}_{A}$ associated to the
skew-symmetrizable matrix $A$ is by definition the subalgebra of
$\mathcal{F}$ generated by all $u_i$ in $\underline{u}$ such that
$(\underline{u}, A')$ is obtained from $(\underline{x}, A)$ by mutations for some $A'$. Such
$\underline{u}=(u_1,u_2,\cdots, u_n)$ is called a cluster of the
cluster algebra $\mathcal{A}_{A}$ or simply of the matrix $A$, and
any $u_i$ is called a cluster variable. The seed
$(\underline{x}, A)$ is called an initial seed. The set  of all cluster
variables is  denoted by $\chi _A .$  If the set $\chi _A$ is
finite, then the cluster algebra $\mathcal{A}_{A}$ is called of
finite type. For any skew-symmetrizable integer matrix $A$, one can
define the Cartan part $C_A$ of $A$ as follow: $C_A=(c_{ij})_{n\times n}$, where
\[c_{ij}=\left\{ \begin{array}{lccccl}-|a_{ij}|&&&&&\mbox{if } i\not=j,\\
  2&&&&&
  \mbox{if }i=j.\end{array}\right. \]
It was proved by Fomin-Zelevinsky \cite{FZ3} that cluster algebras are
of finite type if and only if there is a seed $(\underline{u},A')$ obtained
from the initial seed $(\underline{x},A)$ by mutations such that the Cartan
part $C_{A'}$ of $A'$ is of finite type. In this case, the type of
the Cartan matrix $C_{A}$ is called the type of the cluster algebra
$\mathcal{A}_A$. For example, if \[C_{A'}=
\left(\begin{array}{ccccc}2&-1&0&\cdots& 0\\
-2&2&-1&\cdots&0\\
&&\cdots&\cdots&\\
0&0&\cdots&2&-1\\
 0&&\cdots&-1&2\end{array}\right),\] then the cluster algebra is called of type $C_n$.
\medskip

Now we recall some basics on $2-$Calabi-Yau triangulated categories.
Fix an algebraically closed field $k$. A triangulated category $\C$ is called
$k-$linear provided all Hom-spaces in $\C$ are $k-$spaces and the compositions of
 maps are $k-$linear. The $k-$linear triangulated categories in this paper will be assumed Hom-finite and Krull-Remak-Schmidt, i.e.
  $\dim _k\Hom_\C(X,Y)<\infty $ for any two objects $X$ and $Y$ in $\C$, and every object decomposes into a finite direct sum of objects having
local endomorphism rings. We fix some notations. For an object $M$ in $\C_n$, denote by add$M$ the subcategory of
$\C_n$ consisting of (finite) direct sums of direct summands of $M$.
For two subcategories $\D_1, \D_2$ of $\C_n$, denote by $\D_1*\D_2$
the full subcategory of $\C_n$ consisting of object $E$ such that
there is a triangle $D_1\rightarrow E\rightarrow D_2\rightarrow D_1[1]$, where $D_i\in \D_i$, for $i=1,2.$

A $k-$linear triangulated category $\C$ is called $2-$Calabi-Yau
if there is a functorial isomorphism
$\Hom_\C(X,Y)\cong D\Hom_\C(Y,X[2])$ for any objects $X,Y\in \C $, where
$D=\Hom_k(-,k)$. The main examples of $2-$Calabi-Yau triangulated categories from representation theory of algebras are the cluster
categories of abelian hereditary categories with tilting objects \cite{BMRRT,Ke1};
and the Hom-finite generalized cluster categories of algebras with
global dimension of at most $2$ \cite{Am}; the stable categories of Cohen-Macaulay modules
\cite{BIRS}, cluster tubes \cite{BKL} and some others, please see the survey \cite{Ke3}.

Cluster tilting objects are defined first in cluster categories \cite{BMRRT}, which are generalized to
arbitrary $2-$Calabi-Yau triangulated categories by Keller and Reiten in \cite{KR}.

\begin{defn}
Let $T$ be an object of a $2-$Calabi-Yau triangulated category $\C$.
\begin{enumerate}[4]
\item $T$ is called basic if any two indecomposable summands of $T$ are not isomorphic.
\item $T$ is rigid provided $\Ext_\C^1(T,T)=0$.
\item $T$ is maximal rigid provided $T$ is rigid and is maximal
with respect to this property, i.e. if $\Ext_\C^1(T\oplus M,T\oplus M)=0$,
then $M\in\add T$.
\item $T$ is cluster-tilting provided for any $M\in \C$, $M\in \add T$ if and only if $\Ext_C^1(M,T)=0$.
\end{enumerate}
\end{defn}

From the definition, any cluster tilting object is maximal rigid,
but the converse is not true. It was proved in \cite{BMV} that the
cluster tube $\C_n$ of rank $n$ ($n>1$) has no cluster tilting objects, but maximal
rigid objects.  See \cite{BIKR} for more such examples. The $2-$Calabi-Yau triangulated
categories with cluster tilting objects are important for the
categorification of cluster algebras of skew-symmetric matrices, see
the survey \cite{Ke2,Re} and the references there.

Fix a basic maximal rigid object $T=T_1\oplus \cdots\oplus T_n$ with all $T_j$ indecomposable. For an
$i\in\{1,\cdots,n\}$, write $\bar{T}=\oplus
_{j\not= i}T_j$. Then there are two
non-split triangles:

\[\begin{array}{c}T_i \s{f_i}{\rightarrow}E_i\s{g_i}{\rightarrow}T^*_i \rightarrow T_i[1],\\
T^*_i\s{f'_i}{\rightarrow}E'_i\s{g'_i}{\rightarrow}T_i\rightarrow T^*_i[1],\end{array}\]

where $f_i$ and $f'_i$ are minimal left $\bar{T}-$approximations;
$g_i$ and $g'_i$ are minimal right $\bar{T}-$approximations.
Furthermore $T^*_i$ is indecomposable and $\bar{T}\oplus T^*_i$ is
maximal rigid \cite{IY,BIRS}. Define the mutation of maximal rigid
object $T$ in direction $i$ to be $\mu_i(T)=\bar{T}\oplus T_i$.
It is easy to see that $\mu_i\circ\mu_i(T)=T$. The two triangles
above are called exchange triangles. We define an integer matrix
$A_T=(a_{ij})$ as follows: \[a_{ij}=\alpha_{ij}-\alpha'_{ij},\] where $\alpha_{ij}$ denotes the
multiplicity of $T_i$ as a direct summand of $E_j$, $\alpha'_{ij}$ denotes the
multiplicity of $T_i$ as a direct summand of $E_j'$. Note that
$a_{ii}=0$. Our definition of the matrix $A_T$ associated to $T$ is the
transpose of the matrix $B_T$ defined in \cite{BMV}.
When the endomorphism algebra of $T$ contains no loops or
$2-$cycles, $A_T, B_T$ are skew-symmetric matrices, and
$A_T=-B_T$. In general the matrices $A_T, B_T$ are
sign-skew-symmetric (see Lemma~1.2 in \cite{BMV}).

\begin{rem}
We use our definition of the matrix $A_T$ associated to $T$ to
replace the matrix $B_T$ defined in \cite{BMV} since we will use the
mutation formula $u_ku'_k=\prod_{a_{ik}>0}u_i^{a_{ik}}+\prod_{a_{ik}<0}u_i^{-a_{ik}}$ for the definition of cluster
algebras, where the index $a_{ik}$ comes from the $k-$th column of
the matrix $A$ (see \cite{FZ1}).
\end{rem}

Let $\C$ be a $2-$Calabi-Yau triangulated category with maximal rigid
objects. Suppose that for all maximal rigid objects,
$E_i$ and $E'_i$ have no common direct summands for any $i\in \{1,\cdots,n\}.$ Then
$\mu_i(A_T)=A_{\mu_i(T)}$ (equivalent to $\mu_i(B_T)=B_{\mu_i(T)}$ proved in \cite{BMV}).
In this case one say that the maximal rigid objects form a cluster structure in $\C$ \cite{BMV}.

\bigskip

In what follows, we will focus on cluster tubes, special $2-$Calabi-Yau
triangulated categories. We will denote the tube of rank $n$ by $\T_n$ , where $n$ is always
assumed to be greater than $1$. One realization of this category is the
category of finite-dimensional nilpotent representations over $k$ of
the cyclic quiver $\overrightarrow{\Delta}_{n}$ with $n$ vertices
such that arrows are going from $i$ to $i+1$ (taken modulo $n$). It
is a $k-$linear hereditary abelian category which is Hom-finite, i.e.
dimHom$_{\T_n}(X,Y)<\infty$ for any $X,Y\in\T_n$. Each indecomposable representation is
uniserial, i.e. it has a unique composition series, and hence is
determined by its socle and its length up to isomorphism. We denote
by $(a,b)$ in $\T_n$ the unique indecomposable object with socle
(a,1) and quasi-length b, where $(a,1)$ is the simple representation
at vertex $a$, $a\in \{1,\cdots,n\}$ (see Figure \ref{Fig:tube}).
For convenience, $(a,0)$ denote a zero object. $\T_n$
has Auslander-Reiten sequences, and the Auslander-Reiten translation
$\tau$ is an automorphism of $\T_n$: \[\tau (a, b)=(a-1,b).\]

\begin{figure}[h]
\centering
\setlength{\unitlength}{0.75cm}
\begin{picture}(12,5)
\m(5.1,3.4)(1,-1){1}{\vector(1,-1){0.9}}
\m(3.1,3.4)(1,-1){3}{\vector(1,-1){0.9}}
\m(2.1,2.4)(1,-1){2}{\vector(1,-1){0.9}}

\m(7.95,2.4)(1,-1){2}{\vector(1,-1){0.9}}
\m(8.95,3.4)(1,-1){1}{\vector(1,-1){0.9}}


\m(4.1,0.6)(1,1){2}{\vector(1,1){0.9}}
\m(2.1,0.6)(1,1){3}{\vector(1,1){0.9}}
\m(2.1,2.6)(1,1){1}{\vector(1,1){0.9}}
\m(7.95,0.6)(1,1){2}{\vector(1,1){0.9}}
\m(7.95,2.6)(1,1){1}{\vector(1,1){0.9}}

\m(2,0.5)(0,1){5}{\line(0,1){0.5}} \m(9.9,0.5)(0,1){5}{\line(0,1){0.5}}

\m(3,4.8)(0,-0.5){3}{$\cdot$} \m(5,4.8)(0,-0.5){3}{$\cdot$} \m(8.85,4.8)(0,-0.5){3}{$\cdot$}
\m(6.4,0.4)(0.5,0){3}{$\cdot$} \m(6.4,2.4)(0.5,0){3}{$\cdot$}  \m(6.4,1.4)(0.5,0){3}{$\cdot$}

\put(2.1,0.4){\tiny{$(1,1)=T_{n-1}$}} \put(4.1,0.4){\tiny{$(2,1)$}} \put(6.1,0.2){\tiny{$(3,1)$}}
\put(3.1,1.4){\tiny{$(1,2)=T_{n-2}$}} \put(5.1,1.4){\tiny{$(2,2)$}}
\put(2.1,2.4){\tiny{($n$,3)}} \put(4.1,2.4){\tiny{(1,3)$=T_{n-3}$}} \put(6.1,2.2){\tiny{(2,3)}}

\put(7.95,0.4){\tiny{$(n,1)$}} \put(9.95,0.4){\tiny{$(1,1)$}} \put(8.95,1.4){\tiny{$(n,2)$}}
\put(7.95,2.4){\tiny{$(n-1,3)$}} \put(9.95,2.4){\tiny{$(n,3)$}}

\m(3.05,3.5)(2,0){2}{\circle{0.1}}
\m(3.05,1.5)(2,0){2}{\circle{0.1}}
\m(2.05,2.5)(2,0){3}{\circle{0.1}}
\m(2.05,0.5)(2,0){3}{\circle{0.1}}

\m(7.9,0.5)(2,0){2}{\circle{0.1}}
\m(8.9,1.5)(2,0){1}{\circle{0.1}}
\m(7.9,2.5)(2,0){2}{\circle{0.1}}
\m(8.9,3.5)(2,0){1}{\circle{0.1}}

\end{picture}
\caption{The tube of rank $n$.}
\label{Fig:tube}
\end{figure}
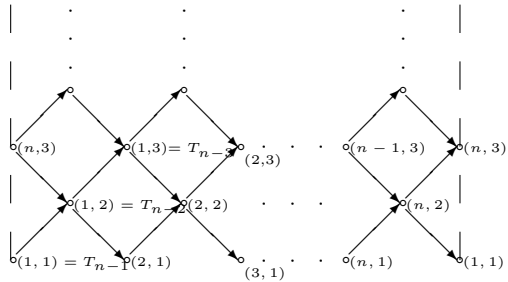

The cluster tube of rank $n$ is defined in \cite{BKL}, as
the orbit category \[\C_n := D^b(\T_n)/\tau ^{-1}[1],\] where $[1]$ is
the shift functor of $D^b(\T _n)$. This category is a $2-$Calabi-Yau triangulated
category such that the projection $\pi: D^b(\T _n)\rightarrow \C_n$
is a triangle functor \cite{BKL,Ke1}. The cluster tube $\C _n$ has Auslander-Reiten triangles induced from the ones in
$D^b(\T_n)$. It is easy to see that indecomposable objects in $\T_n$
are also indecomposable in $\C_n$ (via the composition of the
inclusion functor $\T_n\hookrightarrow D^b(\T_n)$ with the projection
$\pi: D^b(\T _n)\rightarrow \C_n$) and all indecomposable objects in
$\C_n$ are of this form. So we use the same $(a,b)$ to denote the
indecomposable object in $\C_n$ induced from the object $(a,b)$ in $\T_n$.

\medskip

By the definition of $\C_n$, for two objects $X,Y\in\T_n$,
\[\Hom_{\C_n}(X,Y)\cong\Hom_{\T_n}(X,Y)\oplus D\Hom_{\T_n}(Y,\tau^2X).\]
As in \cite{BMV}, the maps from $\Hom_{\T_n}(X,Y)$ are called $\T-$maps and the maps from
$D\Hom_{\T_n}(Y,\tau^2X)$ are called $\D-$maps. Any map from
X to Y in $\C_n$ can be written as the sum of a $\T-$map and a $\D-$map. One knows that the composition of two $\T-$maps is
also a $\T-$map, the composition of a $\T-$map and a $\D-$ map is a $\D-$map, and the composition
of two $\D-$maps is zero.

\medskip

The indecomposable rigid objects are classified in \cite{BMV}:

\begin{prop}
 $(a,b)$ is rigid if and only if $b\leq n-1$.
\end{prop}

Denote by $T_i=(1,n-i)$, $i=1,\cdots, n-1$. It is easy to see that
$\oplus_{i=1}^{n-1}T_i$ is a maximal rigid object in $\C_n$. We will
use $T$ to denote this maximal rigid object through the paper. Let $\D=\add T[-1]*\add T$.
Following \cite{V,Y}, the set of indecomposable objects in $\D[1]$ is
the set of indecomposable objects $(a, b)$ satisfying either (1)
(a,b) is rigid, or (2) $n\leq b\leq 2n-2$ and $a+b\leq 2n-1$.  We
divide $\D$ into five subsets (see Figure \ref{Fig:five}):
 \[\begin{array}{rcl}O&=&\{(a,b)\mid a=1,b\leq n-1\},\\
 I&=&\{(a,b)\mid 2\leq a\leq n-1,a+b\leq n\},\\
 II&=&\{(a,b)\mid a+b\geq n+1, b\leq n-1\},\\
III&=&\{(a,b)\mid a+b\leq 2n-1, a\neq1, b\geq n\},\\
IV&=&\{(a,b)\mid a+b=2n, a\neq1, b\geq n\}.\end{array}\]

\begin{figure}[h]
\centering
\setlength{\unitlength}{0.75cm}
\begin{picture}(18,7.5)

\put(2.3,0.4){\tiny{b=1}} \put(2.3,3.4){\tiny{b=n-1}}
\put(2.3,3.7){\tiny{b=n}}

\put(3,0.5){\line(1,0){0.45}} \put(3.55,0.5){\line(1,0){0.5}}
\put(4.15,0.5){\line(1,0){5.3}} \put(9.55,0.5){\line(1,0){0.5}}
\put(10.15,0.5){\line(1,0){5.35}}

\put(3,3.5){\line(1,0){3.45}} \put(6.55,3.5){\line(1,0){0.5}}
\put(7.15,3.5){\line(1,0){5.9}} \put(13.15,3.5){\line(1,0){2.35}}

\put(3,3.8){\line(1,0){4.35}} \put(7.45,3.8){\line(1,0){5.3}}
\put(12.85,3.8){\line(1,0){0.5}} \put(13.45,3.8){\line(1,0){2.05}}

\put(3.5,0.5){\circle{0.1}} \put(6.5,3.5){\circle{0.1}}
\put(6.8,3.2){\circle{0.1}} \put(4.1,0.5){\circle{0.1}}
\put(7.1,3.5){\circle{0.1}} \put(10.1,0.5){\circle{0.1}}
\put(13.1,3.5){\circle{0.1}} \put(9.5,0.5){\circle{0.1}}

\put(3.55,0.55){\line(1,1){2.9}} \put(6.85,3.15){\line(1,-1){2.6}}
\put(4.15,0.55){\line(1,1){2.6}} \put(7.15,3.45){\line(1,-1){2.9}}
\put(10.15,0.55){\line(1,1){2.9}}

\put(7.4,3.8){\circle{0.1}} \put(10.1,6.5){\circle{0.1}}
\put(12.8,3.8){\circle{0.1}} \put(10.4,6.8){\circle{0.1}}
\put(13.4,3.8){\circle{0.1}}

\put(7.45,3.85){\line(1,1){2.6}} \put(10.15,6.45){\line(1,-1){2.6}}
\put(10.45,6.75){\line(1,-1){2.9}}

\put(3.1,0.25){\tiny{$(1,1)$}} \put(3.8,0.25){\tiny{$(2,1)=T_{n-1}[1]$}}
\put(5.2,3.25){\tiny{$(1,n-1)$}} \put(6.3,2.6){\tiny{$(2,n-2)=T_2[-1]$}}
\put(7.3,3.25){\tiny{$(2,n-1)=T_1[-1]$}} \put(7.5,3.85){\tiny{$(2,n)$}}
\put(8.9,6.45){\tiny{$(2,2n-3)$}} \put(10.45,6.75){\tiny{$(2,2n-2)$}}
\put(11.6,3.85){\tiny{$(n-1,n)$}} \put(13,3.25){\tiny{$(n,n-1)$}}
\put(13.3,3.85){\tiny{$(n,n)$}} \put(8.8,0.25){\tiny{$(n-1,1)$}}
\put(9.9,0.25){\tiny{$(n,1)$}}

\put(5,2){O} \put(6.7,1.5){I} \put(9.9,2.1){II} \put(9.8,4.7){III}
\put(11.5,5.3){IV}

\end{picture}
\caption{Five disjoint subsets of $\D$.}
\label{Fig:five}
\end{figure}

Let $f_1^{(a,b)}$ be a nonzero $\T-$map from $T_1[-1]$ to $(a,b)\in II\cup III\cup IV$ and let $g_1^1$ be a
nonzero $\D-$map from $T_1[-1]$ to itself. Write \[g_1^{(a,b)}=f_1^{(a,b)}g_1^1.\] The triangle involving $g_1^1$ is
\[(3,2n-2)\rightarrow T_1[-1]\s{g_1^1}\rightarrow T_1[-1]\s{f}\rightarrow (2,2n-2)\] where $f=\lambda f^{(2,2n-1)}_1$ for some nonzero $\lambda\in k$. We know that $\alpha f=0$ for any $\T-$map $\alpha$ in
$\Hom_{\C_n}((2,2n-2),(a,b))$ for $(a,b)\in II\cup III$, so $f_1^{(a,b)}$ can not factor through $f$,
and then $g_1^{(a,b)}=f_1^{(a,b)}g_1^1\neq0$.
Computing the dimension of $\Hom_{\C_n}(T_1[-1],(a,b))$, we have the following obvious fact.

\begin{lem}\label{Lem:dimt1} For any indecomposable object $(a,b)$ in $\D$, we have that
\[
\dim\Hom_{\C_n}(T_1[-1],(a,b))=\begin{cases}
0 & \text{if $(a,b)\in O$ or $I$};\\
2 & \text{if $(a,b)\in II$ or $III$;}\\
1 & \text{if $(a,b)\in IV$}.
\end{cases}
\]
Moreover, $f_1^{(a,b)},g_1^{(a,b)}$ form a basis of $\Hom_{\C_n}(T_1[-1],(a,b))$ for $(a,b)\in II$ or $III$;
$f_1^{(a,b)}$ forms a basis of $\Hom_{\C_n}(T_1[-1],(a,b))$ for $(a,b)\in IV$.
\end{lem}

Let $t_i^{i-1}$ be a nonzero $\T-$map from $T_i[-1]$ to $T_{i-1}[-1]$ and $\id_i$ be the identity from $T_i[-1]$ to itself. We write
\[f_i^{(a,b)}=f^{(a,b)}_k\id_kt_{k+1}^k\cdots t_{i-1}^{i-2}t_i^{i-1},\]
and \[g_i^{(a,b)}=g_1^{(a,b)}f_i^1,\]
where $k=1$ if $(a,b)\in II$; $k=n-a-b+2$ and $f_k^{(a,b)}$ is a nonzero $\T-$map
from $T_k[-1]$ to $(a,b)$ if $(a,b)\in I$. Then we know that $f^{(a,b)}_i$ (resp. $g_i^{(a,b)}$) is a nonzero $\T-$map (resp. $\D-$map)
if there exist nonzero $\T-$maps (resp. $\D-$maps) from $T_i[-1]$ to $(a,b)$, since any $\D-$map factors through the ray starting $T_i$ (Lemma~2.2 in \cite{BMV}).
By our setting, we know that $f^{(a,b)}_i\neq f^{(a,b)}_{i'}$ if $i\neq i'$. Computing the dimension of $\Hom_{\C_n}(T_i[-1],(a,b))$,
we have the following obvious fact.

\begin{lem}\label{Lem:dimti} For any indecomposable rigid object $(a,b)$, i.e. $(a,b)\in O\cup I\cup II$, we have that
\[
\dim \Hom_{\C_n}(T_i[-1],(a,b))=\begin{cases}
2& \text{if $1\leq i\leq n-a+1$, $(a,b)\in II$;}\\
1&  \begin{matrix}\text{if $n-a+2\leq i\leq 2n-a-b$, $(a,b)\in II$,\ \ \ \ \ } \\ \text{or $n-a-b+2\leq i\leq n-a+1$, $(a,b)\in I$;}\end{matrix}\\
0& \text{otherwise.}
\end{cases}
\]
Moreover, a basis of $\Hom_{\C_n}(T_i[-1],(a,b))$ is $\{f^{(a,b)}_i, g^{(a,b)}_i\}$ for $1\leq i\leq n-a+1$, $(a,b)\in II$ ;
$\{g^{(a,b)}_i\}$ for $n-a+2\leq i\leq 2n-a-b$, $(a,b)\in II$; and $\{f^{(a,b)}_i\}$ for $n-a-b+2\leq i\leq n-a+1$, $(a,b)\in I$ respectively.
\end{lem}

\section{Index and the cluster map}\label{Sec:cha}

We use the same notations as the above section. In this section, we will
define the cluster map $X_{?}$ from the set of indecomposable rigid
objects in $\C_n$ to $\mathcal{F}=\mathbb{Q}(x_1,x_2,\cdots, x_n)$ by using the geometric
information of the indecomposable rigid objects. We will give an
explicit expression of $X_M$ as a Laurent polynomial of $x_1,
\cdots, x_{n-1}$ according to which subset $M$ belongs to (recall that the
set of indecomposable rigid objects in $\C_n$ is divided into three
disjoint subsets $O$, $I$ and $II$, see Figure \ref{Fig:five}).
\medskip

Let $K_0^{split}(T)$ be
 the split-Grothendieck group of add$T$, i.e. the free abelian group
 with a basis consisting of isomorphism classes $[T_1],\cdots,[T_{n-1}]$ of indecomposable
 direct summands of $T$.
\medskip

For an object $X$ of $\D[1]=\add T*\add T[1]$, there
exists a traingle
\[
T''_X\rightarrow T'_X\s{f}\rightarrow X\rightarrow T''_X[1]
\]
where $T''_X,T'_X\in\add T$. It follows that $f$ is a right
add$T-$approximation. We define the index
\[ind_T(X)=[T'_X]-[T''_X]\in K_0^{split}(T)\] as in \cite{Pa1,DK,Pla1,Pla2}. It
was proved in \cite{ZZ} that any rigid object belongs to $\D[1]$ (also in
$\D$). The next lemma tells us how to get the right
add$T-$approximation of any indecomposable rigid object in $\C_n$.

\begin{lem}
For any indecomposable rigid object $(a,b)$ in $\T_n$, every right $\add T-$approximation of $(a,b)$ in $\T_n$ is a right $\add T-$approximation of $(a,b)$ in $\C_n$.
\end{lem}

\begin{proof}
If there are no non-zero $\D-$maps in
$\Hom_{\C_n}(T,(a,b)),$ then a right $\add T-$approximation of $(a,b)$
in $\T_n$ is a right $\add T-$approximation of $(a,b)$ in $\C_n$. Now
suppose there are non-zero $\D-$maps in $\Hom_{\C_n}(T,(a,b)),$
then by Lemma \ref{Lem:dimt1} and Lemma \ref{Lem:dimti}, anyone of them is a sum of some maps $fg_1^1f_i^1$
where $f$ is a $\T-$map from $T_i$ to $(a,b)$. Therefore,
every $\D-$map in $\Hom_{\C_n}(T,(a,b))$ factors through $\T-$maps from $T$ to $(a,b)$.
Thus we have the assertion.
\end{proof}

We calculate the index of any indecomposable rigid object $(a,b)$ respect to $T$.

\begin{lem}\label{Lem:index}
For any indecomposable rigid object $(a,b)$ in $\C_n$,
\[\ind_T(a,b)=\begin{cases}
[T_1]-[T_{n-a+1}]-[T_{2n-a-b}]& \text{if $a+b\geq n+1$}\\
[T_{n-a-b+1}]-[T_{n-a+1}]& \text{if $a+b\leq n$}.
\end{cases}\]
\end{lem}

\begin{proof}
For $a+b\geq n+1$, there is a minimal right
$\add T-$approximation $f:T_1\rightarrow (a,b)$ in $\T_n$. Then
$\ker(f)=(1,a-1)$ and $\coker(\tau^{-1}f)=(1,a+b-n)$. We get a
triangle
\[
C\rightarrow T_1\s{f}\rightarrow (a,b)\rightarrow C[1]
\]
in $\C_n$ and an exact sequence
\[
0\rightarrow\coker(\tau^{-1}f)\rightarrow C\rightarrow\ker(f)\rightarrow0
\]
in $\T_n$ (cf. \cite{Y}). Since $\Ext_{\T_n}^1((1,a-1),(1,a+b-n))=0$,
then $C\cong(1,a-1)\oplus(1,a+b-n)$. Hence
$\ind_T(a,b)=[T_1]-[T_{n-a+1}]-[T_{2n-a-b}]$.

For $a+b\leq n$, there is a minimal right $\add T-$approximation
$f:T_{n-a-b+1}\rightarrow (a,b)$ in $\T_n$. Then $\ker(f)=(1,a-1)$
and $\coker(\tau^{-1}f)=0$. Hence
$\ind_T(a,b)=[T_{n-a-b+1}]-[T_{n-a+1}]$.
\end{proof}

Let $B=\End_{\C_n}(T[-1])$, then
$F:=\Hom_{\C_n}(T[-1],-)$: $\D/\add T$ $\rightarrow$ $\text{mod}B$ is
an equivalence of abelian categories \cite{IY,Y,ZZ},
where mod$B$ denotes the category of finite dimensional right $B-$modules.
The maps $\id_i$, $1\leq i\leq n-1$, $t_i^{i-1}$, $2\leq i\leq n-1$, and
$g_1^1$ form a set of generators of $B$.

\begin{defn}\label{Defn:cc}
For any indecomposable rigid object $M$ in $\C_n$, we define
\[
X_M=x^{\ind_TM}\sum_{e\in\mathbb{N}^{n-1}}\chi\left(\Gr_e(FM)\right)x^{-\iota(e)}
\]
where the sum takes over all dimension vectors $e$ such that there
exists an object $Y$ in $\D\cap\D[1]$ with
$e=\underline{\dim}FY$; where
$\iota(e)=\ind_TY+\ind_TY[1]$ and
$\chi\left(\Gr_eFM\right)$ is the Euler characteristic
of the quiver Grassmannian of dimension vector $e$ of $FM$ (see \cite{Ke2});
where $x^{\sum_{i=1}^{n-1}a_i[T_i]}=\prod_{i=1}^{n-1}x_i^{a_i}$.

The definition of $X_M$ can be extended to rigid objects: for any
rigid object $N=\bigoplus_{i=1}^mM_i$ in $\C_n$ with $M_i$
indecomposable, we define
\[
X_N=\prod_{i=1}^mX_{M_i}.
\]
\end{defn}

In \cite{Pla1,Pla2}, a similar definition is given respect to a fixed
rigid object $T$ in a $2-$Calabi-Yau triangulated
category with infinite Hom-spaces. The definition there needs an additional
assumption that any
finite-dimensional $B-$module can be lifted through $F$ to an object
in $\D\cap\D[1]$. This assumption is not satisfied in our situation.
For example, the simple $B-$module $S_1$ corresponding to $T_1$ can not
lift to any object in $\D\cap\D[1]$ by the functor $F$
(see the dimension formula in Lemma \ref{Lem:dimt1}). So the definition in
\cite{Pla1,Pla2} can not apply to our case. In our definition, we omit the
$B-$submodules which can not be lifted through $F$ to $\D\cap\D[1]$.

\medskip

The following lemma points out that for an indecomposable rigid
object $M$, if a $B-$submodule of $FM$ can not be lifted to $\D\cap\D[1]$,
then other $B-$submodules of $FM$ with the same dimension can not be lifted to $\D\cap\D[1]$ neither.
So the Euler characteristic of $\Gr_e(FM)$ is well-defined.

\begin{lem}
Let $M$ be an indecomposable rigid object in $\C_n$ and $Y$ be an
object in $\D$ such that $FY$ is a $B-$submodule of $FM$. Then
$Y\in\D\cap\D[1]$ if and only if $\dim\Hom_{\C_n}(T_1[-1],Y)\neq1$.
\end{lem}

\begin{proof}
For an object $Y$ in $\D$, it is easy to see that $Y\in\D\cap\D[1]$ if and only if all indecompasable summands of
$Y$ are in $O\bigcup I\bigcup II\bigcup III$, see Figure \ref{Fig:five}. Since $\dim\Hom(T_1[-1],(a,b))\leq2$, so $\dim\Hom(T_1[-1],Y)=2$ implies
that $\dim\Hom(T_1[-1],(a,b))=2$. In this case $f_1^{(a,b)}$ is a generator of $FM$ as $B-$module by Lemma \ref{Lem:dimti} and $f_1^{(a,b)}\in FY$.
So we have that $FY\cong FM$. Then $Y\cong M\oplus T'$, where $T'\in \mbox{add}T$. Hence $Y \in\D\cap\D[1]$.
When $\dim\Hom(T_1[-1],Y)=1$, $Y$ has some summand in $IV$ by the dimension formula in Lemma \ref{Lem:dimt1}, then $Y\notin\D\cap\D[1]$.
When $\dim\Hom(T_1[-1],Y)=0$, all summands of $Y$ are in $O$ or $I$, then $Y\in\D\cap\D[1]$. Thus the statement holds.
\end{proof}

We determine the submodules of F(a,b) for any indecomposable rigid object $(a,b)$.
Put \[\X(a,b)=\{(a,k)\mid0\leq k\leq \min\{n-a,b\}\}\subset I\] and \[\Y(a,b)=\{(a+b-n+1,l)\mid 0\leq l\leq 2n-a-b-1\}\subset I.\]
We denote by Sub$(a,b)$ the set of submodules of $F(a,b)$, the first terms of whose dimension vectors are not 1.

\begin{lem}\label{Lem:submod}
Let $(a,b)$ be an indecomposable rigid object object in $\D$. Then
\[
\text{Sub}(a,b)=\begin{cases}
\{FX\mid X\in\X(a,b)\} & \text{if $(a,b)\in I$};\\
\{FX\oplus FY\mid X\in\X(a,b),Y\in\Y(a,b)\}\cup\{(a,b)\} & \text{if $(a,b)\in II$}.
\end{cases}
\]
For every dimension vector $e\neq\dim F(a,b)$, whose the first term is not 1, $\chi(\Gr_eF(a,b))$ is equal to
the number of elements in the set sub$_e(a,b)$ where
\[
\text{sub}_e(a,b)=\begin{cases}
\{X\in \X(a,b)\mid \dim FX=e\} & \text{if $(a,b)\in I$};\\
\{(X,Y)\in \X(a,b)\times\Y(a,b)\mid \dim(FX\oplus FY)=e\} & \text{if $(a,b)\in II$}.
\end{cases}
\]
\end{lem}

\begin{proof}
\begin{enumerate}[2]
  \item The case $(a,b)\in I$. By Lemma \ref{Lem:dimti}, a basis of $F(a,b)$ is
  \[\{f^{(a,b)}_i\mid n-a-b+2 \leq i \leq n-a+1\}\]
  which satisfy the conditions
  \[
  \begin{array}{rcl}
  f^{(a,b)}_i\id_j&=&\begin{cases}
  f^{(a,b)}_i & \text{if $j=i$,}\\
  0 & \text{if $j\neq i$;}
  \end{cases}\\
  f^{(a,b)}_it_j^{j-1}&=&\begin{cases}
  f^{(a,b)}_{i+1} & \text{if $j=i+1\leq n-a+1$,}\\
  0 & \text{otherwise;}
  \end{cases}\\
  f^{(a,b)}_ig_1^1&=&0.
  \end{array}
  \]
  For every nonzero submodule $S$ of $F(a,b)$, there exists the
  minimal $i_0$, $n-a-b+2 \leq i_0 \leq n-a+2$ such that
  \[f^{(a,b)}_{i_0}+\lambda_1 f^{(a,b)}_{i_0+1}+\cdots+\lambda_{n-a-i_0+1} f^{(a,b)}_{n-a+1}\in S.\]
  By multiplying with $\id_{i_0}$, we have that the first term $f^{(a,b)}_{i_0}$
  is in $S.$ Let $k=n-a-i_0+2,$ then $0\leq k\leq b$ and a basis of $S$ is
  \[\{f^{(a,b)}_i\mid n-a-k+2 \leq i \leq n-a+1\}.\]
  Consider the module $F(a,k)$ whose basis is
  \[\{f^{(a,k)}_i\mid n-a-k+2 \leq i \leq n-a+1\},\]
  which satisfy the conditions
  \[
  \begin{array}{rcl}
  f^{(a,k)}_i\id_{i'}&=&\begin{cases}
  f^{(a,k)}_i & \text{if $i'=i$,}\\
  0 & \text{if $j\neq i$;}
  \end{cases}\\
  f^{(a,k)}_it_{i'}^{i'-1}&=&\begin{cases}
  f^{(a,k)}_{i+1} & \text{if $i'=i+1\leq n-a+1$,}\\
  0 & \text{otherwise;}
  \end{cases}\\
  f^{(a,k)}_ig_1^1&=&\ \ 0.
  \end{array}
  \]
  So there is a natural isomorphism of modules:
  \[\begin{matrix}S&\longrightarrow & F(a,n-a-k+2)\\ f^{(a,b)}_i & \mapsto & f^{(a,k)}_i\end{matrix}.\]
  Clearly, every dimension vector corresponds to one submodule and $\chi(\Gr_eF(a,b))=1$. We finish the proof in this case.

  \item The case $(a,b)\in II$. In this case, by Lemma \ref{Lem:dimti}, there is a basis
  \[\{f^{(a,b)}_i, g_j^{(a,b)}\mid 1\leq i\leq n-a+1, 1\leq j \leq 2n-a-b\}\]
  of $F(a,b)$ and
  \[
  \begin{array}{rcl}
  f^{(a,b)}_i\id_{i'}&=&\begin{cases}
  f^{(a,b)}_i & \text{if $i'=i$,}\\
  0 & \text{if $i'\neq i$;}
  \end{cases}\\
  f^{(a,b)}_it_{i'}^{i'-1}&=&\begin{cases}
  f^{(a,b)}_{i+1} & \text{if $i'=i+1\leq n-a+1$,}\\
  0 & \text{otherwise;}
  \end{cases}\\
  f^{(a,b)}_ig_1^1&=&\begin{cases}
  g_1^{(a,b)} & \text{if $i=1$,}\\
  0 & \text{if $i\neq 1$;}
  \end{cases}\\
  g^{(a,b)}_j\id_{j'}&=&\begin{cases}
  g^{(a,b)}_j & \text{if $j'=j$,}\\
  0 & \text{if $j'\neq j$;}
  \end{cases}\\
  g^{(a,b)}_jt_{j'}^{j'-1}&=&\begin{cases}
  g^{(a,b)}_{j+1} & \text{if $j'=j+1\leq 2n-a-b$,}\\
  0 & \text{otherwise;}
  \end{cases}\\
  g^{(a,b)}_jg_1^1&=&\ \ 0.
  \end{array}
  \]
  For a submodule $S\in$ Sub$(a,b)$, if the first term of dimension vector of $S$ is not 0, then $f_1^{(a,b)}\in S$. Since $f_1^{(a,b)}$ is a generator of $F(a,b)$, $S=F(a,b)$.
  If $f_1^{(a,b)},g_1^{(a,b)}\notin S$, by a similar discussion as above, there is a minimal $i_0$, $2\leq i_0\leq n-a+2$,
  (resp. a minimal $j_0$, $2\leq j_0\leq 2n-a-b+1$) such that $f^{(a,b)}_{i_0}$ is in $S$
  (resp. $\lambda f^{(a,b)}_{j_0}+g^{(a,b)}_{j_0}$ is in $S$ for some $\lambda\in k$).
  Let $k=n-a-i_0+2,$ $l=2n-a-b-j_0+1,$ then $0\leq k\leq n-a,$ $0\leq l\leq 2n-a-b-1,$ and a basis of $S$ is
  \[\{f^{(a,b)}_{i}, \lambda f^{(a,b)}_{j}+g^{(a,b)}_{j}\mid \begin{matrix}n-a-k+2\leq i\leq n-a+1,\\ 2n-a-b-l+1\leq j\leq 2n-a-b\end{matrix}\}.\]
  Consider the module $F(a,k)\oplus F(a+b-n+1,l)$ whose basis is
  \[\{f^{(a,k)}_i, f^{(a+b-n+1,l)}_j\mid \begin{matrix}n-a-k+2 \leq i \leq n-a+1,\\ 2n-a-b-l+1 \leq j \leq 2n-a-b \end{matrix}\},\]
  which satisfy the conditions
  \[
  \begin{array}{rcl}
  f^{(a,k)}_i\id_{i'}&=&\begin{cases}
  f^{(a,k)}_i & \text{\ \ \ \ \ \ \ \ \ if $i'=i$,}\\
  0 & \text{\ \ \ \ \ \ \ \ \ if $i'\neq i$;}
  \end{cases}\\
  f^{(a,k)}_it_{i'}^{i'-1}&=&\begin{cases}
  f^{(a,k)}_{i+1} & \text{\ \ \ \ \ \ \ \ \ if $i'=i+1\leq n-a+1$,}\\
  0 & \text{\ \ \ \ \ \ \ \ \ otherwise;}
  \end{cases}\\
  f^{(a,k)}_ig_1^1&=&\ \ 0;\\
  f^{(a+b-n+1,l)}_j\id_{j'}&=&\begin{cases}
  f^{(a+b-n+1,l)}_j & \text{if $j'=j$,}\\
  0 & \text{if $j'\neq j$;}
  \end{cases}\\
  f^{(a+b-n+1,l)}_jt_{j'}^{j'-1}&=&\begin{cases}
  f^{(a+b-n+1,l)}_{j+1} & \text{if $j'=j+1\leq 2n-a-b$,}\\
  0 & \text{otherwise;}
  \end{cases}\\
  f^{(a+b-n+1,l)}_jg_1^1&=&\ \ 0.
  \end{array}
  \]
  So there is a natural isomorphism of modules:
  \[\begin{matrix}S&\longrightarrow & F(a,n-a-k+2)\\ f^{(a,b)}_i & \mapsto & f^{(a,k)}_i\\
  \lambda f^{(a,b)}_{j}+g^{(a,b)}_{j}&\mapsto & f^{(a+b-n+1,l)}_j\end{matrix}.\]
  Using the same notations above, to get the formula of the Euler character, we only need to prove that the Euler character of the Grassmannian associated to the submodule $S$ above with all values of $\lambda$ chosen to be 1, since any $\lambda$ offer the same submodule up to isomorphism and $\chi$ is additive with respect to disjoint unions. If $i_0\geq j_0$, then the associated Grassmannian contains only one point, so its Euler character is 1; if $i_0<j_0$, then the associated Grassmannian is $\mathbb{P}^1\setminus\{\text{a point}\}$, so its Euler character is also 1. Thus we complete the proof.
\end{enumerate}
\end{proof}

Before giving the explicit formula of $X_{(a,b)}$, we need show why $\iota(e)$ does not depend on the choice of $Y$.
Define \[\iota(Y)=\ind_TY+\ind_TY[1].\] Clearly, $\iota(Y_1\oplus Y_2)=\iota(Y_1)+\iota(Y_2)$ and $\iota(T')=0$ for any $T'\in\add T$. If there are two submodules $FY\cong FY'$, then $Y\oplus T_1\cong Y'\oplus T_2$ with $T_1,T_2\in\add T$. So $\iota(Y)=\iota(Y')$. If there are two non-isomorphic submodules $FY$, $FY'$ with the same dimension, by Lemma \ref{Lem:submod} and the discussion above,
we can write $Y=(a,k)\oplus (a+b-n+1,l)$ and $Y'=(a,k')\oplus (a+b-n+1,l')$. Then \[\dim FY= \sum_{i=n-a-k+2}^{n-a+1}e_i+\sum_{j=2n-a-b-l+1}^{2n-a-b}e_j=\dim FY'=\sum_{i=n-a-k'+2}^{n-a+1}e_i+\sum_{j=2n-a-b-l'+1}^{2n-a-b}e_j,\] where $e_i$ is the dimension vector of the simple module of $B$ corresponding to $T_i$. Note that $k=k'$ implies $l=l'$ and $FY$ is not isomorphic to $FY'$, so $k\neq k'$. Without loss of generality, we assume that $k> k'$. Then $k'=l-n+b+1$, $l'=k+n-b-1$. By Lemma $\ref{Lem:index}$, we have that \[\iota((a,k)\oplus (a+b-n+1,l))=[T_{n-a-k+1}]-[T_{n-a+1}]+[T_{2n-a-b-l}]-[T_{2n-a-b}],\] and
\[\iota((a,k')\oplus (a+b-n+1,l'))=[T_{2n-a-b-l}]-[T_{n-a+1}]+[T_{n-a-k+1}]-[T_{2n-a-b}].\] So $\iota(Y)=\iota(Y')$. Combining with Lemma \ref{Lem:submod}, we have a new form of Definition \ref{Defn:cc}.

\begin{lem}\label{Lem:newform}
For any indecomposable rigid object $(a,b)$ in $\C_n$, we have that
\[
X_{(a,b)}=\begin{cases}
x^{\ind_T(a,b)} & \text{if $(a,b)\in O$};\\
x^{\ind_T(a,b)}\sum\limits_{X\in\X(a,b)}x^{-\iota(X)} & \text{if $(a,b)\in I$};\\
x^{\ind_T(a,b)}\left(x^{-\iota(a,b)}+\sum\limits_{X\in\X(a,b),Y\in\Y(a,b)}x^{-\iota(X\oplus Y)}\right) & \text{if $(a,b)\in II$}.
\end{cases}
\]
\end{lem}

Now we give the explicit expression of $X_{(a,b)}$ for each indecomposable rigid object $(a,b)$. In the following
expression of $x_{m}$, when $m=n$, $x_m$ is taken to be 1. In this setting, the first formula in Lemma \ref{Lem:index}
also holds in the case of $a+b=n$.

\begin{thm}\label{Thm:1}
For any indecomposable rigid object $(a,b)$, we have that
\[
X_{(a,b)}=\begin{cases}
x_{n-b} & \text{if $(a,b)\in O$};\\
\sum\limits_{k=0}^b\frac{x_{n-a-b+1}x_{n-a+2}}{x_{n-a-k+1}x_{n-a-k+2}}& \text{if $(a,b)\in I$;}\\
x_1x_{K+2}x_{L+2}\left(\frac{1}{x_1^2}
+\left(\sum\limits_{k=0}^{K}\frac{1}{x_{K-k+1}x_{K-k+2}}\right)
\left(\sum\limits_{l=0}^{L}\frac{1}{x_{L-l+1}x_{L-l+2}}\right)\right)& \text{if $(a,b)\in II$};
\end{cases}
\]
where $K=n-a$, $L=2n-a-b-1$.
\end{thm}

\begin{proof}

\begin{enumerate}
\item The case of $(a,b)\in O$. $X_{(a,b)}=x^{\ind_T(a,b)}=x_{n-b}$.

\item The case of $(a,b)\in I$. By Lemma \ref{Lem:newform}, we have that
\[\begin{array}{rcl}X_{(a,b)}&=&x^{\ind_T(a,b)}\sum\limits_{X\in\X(a,b)}x^{-\iota(X)}\\
&=&x^{\ind_T(a,b)}\sum\limits_{k=0}^bx^{-\iota(a,k)}\\
&=&\frac{x_{n-a-b+1}}{x_{n-a+1}}\sum\limits_{k=0}^b\frac{x_{n-a+1}x_{n-a+2}}{x_{n-a-k+1}x_{n-a-k+2}}\\
&=&\sum\limits_{k=0}^b\frac{x_{n-a-b+1}x_{n-a+2}}{x_{n-a-k+1}x_{n-a-k+2}}.\end{array}\]

\item The case of $(a,b)\in II$. By Lemma \ref{Lem:newform}, we have that
\[\begin{array}{rcl}X_{(a,b)}&=&x^{\ind_T(a,b)}\left(x^{-\iota(a,b)}+\sum\limits_{X\in\X(a,b),Y\in\Y(a,b)}x^{-\iota(X\oplus Y)}\right)\\
&=&x^{\ind_T(a,b)}\left(x^{-\iota(a,b)}+\left(\sum\limits_{k=0}^bx^{-\iota(a,k)}\right)\left(\sum\limits_{l=0}^{2n-a-b-l}x^{-\iota(a+b-n+1,l)}\right)\right)\\
&=&x_1x_{K+2}x_{L+2}\left(\frac{1}{x_1^2}
+\left(\sum\limits_{k=0}^{K}\frac{1}{x_{K-k+1}x_{K-k+2}}\right)
\left(\sum\limits_{l=0}^{L}\frac{1}{x_{L-l+1}x_{L-l+2}}\right)\right).\end{array}\]

\end{enumerate}
\end{proof}

\section{Mutation relations}\label{Sec:mu}

By Proposition 2.6 in \cite{BMV}, every maximal rigid object in $\C_n$ is
in some wing of $(a,n-1)$ and there is a natural bijection between
the set of maximal rigid object in the wing of $(a,n-1)$ and the set
of tilting modules over the path algebra $k\vec{A}_{n-1}$ of the linear
quiver of type $A_{n-1}$. Hence using the complete description of
all tilting modules of quivers of type $A$ in \cite{HR}, we can get all
maximal rigid objects by the following induction starting with an chosen object $(a,n-1)$:
for each chosen object $(a,b)$, one choose two objects $(a,h-1)$ and $(a+h,b-h)$ until all the new object are in the bottom of the tube,
see Figure \ref{Fig:mr}; take the direct sum of all chosen objects.

\begin{figure}[h]
\centering
\includegraphics[totalheight=1.4in]{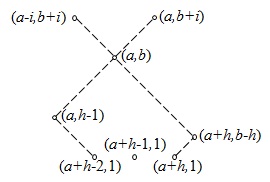}
\caption{The construction of maximal rigid objects.}
\label{Fig:mr}
\end{figure}

Under this construction of maximal rigid objects, we have the following fact.

\begin{lem}\label{Lem:conofmr}
Let $R$ be a maximal rigid object in $\C_n$ and $(a,b)$ be an
indecomposable summand of $R$. Then there are $(a,h-1)$ and
$(a+h,b-h)$ in $\add R$ for some $h$, $1\leq h\leq b$ and every
indecomposable summand of $R$ in the wing of $(a,b)$ which is not
isomorphic to $(a,b)$ is in the wing of $(a,h-1)$ or in the wing of
$(a+h,b-h)$. When $b\neq n-1$, there exists an $i$, $1\leq i\leq
n-b-1$, such that either $(a,b+i)$ or $(a-i,b+i)$ is in $\add R$. See Figure \ref{Fig:mr}.
\end{lem}

By this lemma, we can give all exchange triangles in $\C_n$.

\begin{lem}\label{Lem:mutofmr}
Given two basic maximal rigid objects $T'\oplus\overline{T}$ and
$T''\oplus\overline{T}$ in $\C_n$ such that both $T'$ and $T''$ are
indecomposable. Then $\dim\Ext_{\C_n}^1(T',T'')=1$ or $2$. Moreover, if $\dim\Ext_{\C_n}^1(T',T)=2$, then the exchange
triangles have the following form:
\[
(a,n-1)\rightarrow(a+h,n-h-1)\oplus(a+h,n-h-1)\rightarrow(a+h,n-1)\rightarrow (a,n-1)[1],
\]
\[
(a+h,n-1)\rightarrow(a,h-1)\oplus(a,h-1)\rightarrow(a,n-1)\rightarrow (a+h,n-1)[1],
\]
where $1\leq a\leq n$, $1\leq h\leq n-1$. If $\dim\Ext_{\C_n}^1(T',T)=1$, then the exchange triangles have the
following form:
\[
(a,b)\rightarrow (a,b+i)\oplus(a+h,b-h)\rightarrow(a+h,b-h+i)\rightarrow (a,b)[1],
\]
\[
(a+h,b-h+i)\rightarrow (a+b+1,i-1)\oplus(a,h-1)\rightarrow(a,b)\rightarrow (a+h,b-h+i)[1],
\]
where $1\leq a\leq n$, $b\leq n-2$, $1\leq h\leq b$, $1\leq i\leq
n-b-1$.
\end{lem}

\begin{proof}
Let $(a,n-1)\oplus\overline{R}$ be a basic maximal rigid object,
$1\leq a\leq n$. Then by Lemma \ref{Lem:conofmr}, there is an $h$, $1\leq h\leq n-1$, such that
$(a,h-1)$ and $(a+h,n-h-1)$ are in $\add \overline{R}$
and other indecomposable summands of $\overline{R}$ are in the wing
of $(a,h-1)$ or in the wing of $(a+h,n-h-1)$. See Figure \ref{Fig:mutofmr1}. Now
$(a+h,n-1)\oplus\overline{R}$ is a basic maximal rigid object.

\begin{figure}[h]
\centering
\includegraphics[totalheight=1.4in]{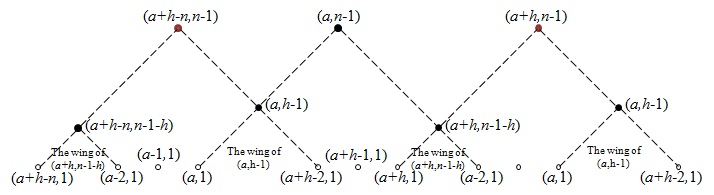}
\caption{Mutations of maximal rigid objects as in the first part of the proof of Lemma \ref{Lem:mutofmr}.}
\label{Fig:mutofmr1}
\end{figure}

Consider a nonzero $\T-$map $f:(a+h,n-1)\rightarrow (a,n-1)[1]=(a-1,n-1)$, we have a triangle in $\C_n$:
\[C\rightarrow (a+h,n-1)\s{f}\rightarrow (a,n-1)[1]\rightarrow C[1]\]
with an exact sequence
\[0\rightarrow \coker(\tau^{-1}f)\rightarrow C\rightarrow \ker(f)\rightarrow 0\]
in $\T_n$ (cf. \cite{Y}). Since $\ker(f)=(a+h,n-1-h)$, $\coker(\tau^{-1}f)=(a+h,n-1-h)$, and
$(a+h,n-1-h)$ is rigid in $\T_n$, we have that $C\cong (a+h,n-1-h)\oplus (a+h,n-1-h)$.
Thus we get a triangle in $\C_n$:
\[(a,n-1)\rightarrow (a+h,n-1-h)\oplus (a+h,n-1-h)\rightarrow (a+h,n-1)\s{f}\rightarrow (a,n-1)[1]\]
and it is the exchange triangle starting at $(a,n-1)$ associated to $\overline{R}$
since $(a+h,n-1)$ is another complement of $\overline{R}$ and $(a+h,n-1-h)\in\add \overline{R}$.
The exchange triangle starting at $(a+h,n-1-h)$ associated to $\overline{R}$ can be obtained similarly.
Clearly, $\dim\Ext_{\C_n}^1((a,n-1),(a+h,n-1))=2$.

\medskip

Let $(a,b)\oplus\overline{R}$ be a basic maximal rigid
object with $1\leq a\leq n$, $b\leq n-2$. By Lemma \ref{Lem:conofmr},
we can assume that there is the minimum $i$, $1\leq i\leq n-b-1$,
such that $(a,b+i)\in \overline{R}$. By the construction of maximal
objects, $(a+b+1,i-1)\in\overline{R}$, see Figure \ref{Fig:mutofmr2}.
Then $(a+h,b-h+i)\oplus\overline{R}$ is a basic maximal rigid object.
Clearly, $\dim\Ext_{\C_n}^1((a,b),(a+h,b-h+i))=1$.

\begin{figure}[h]
\centering
\includegraphics[totalheight=1.4in]{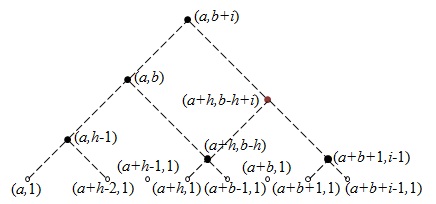}
\caption{Mutations of maximal rigid objects as in the second part of the proof of Lemma \ref{Lem:mutofmr}.}
\label{Fig:mutofmr2}
\end{figure}

There is an obvious non-split triangle in $\C_{n}$:
\[(a,b)\rightarrow (a,b+i)\oplus(a+h,b-h)\rightarrow(a+h,b-h+i)\rightarrow (a,b)[1],\]
which is the exchange triangle starting at $(a,b)$ associated to $\overline{R}$.
To get the other exchange triangle, we consider a nonzero $\T-$map $f:(a,b)\rightarrow(a+h,b-h+i)[1]=(a+h-1,b-h+i)$.
We have that $\ker(f)=(a,h-1)$ and $\coker(\tau^{-1}f)=(a+b+1,i-1)$ in $\T_n$.
Since $\Ext_{\T_n}^1((a,h-1),(a+b+1,i-1))=0$, we get a non-split triangle in $\C_n$:
\[(a+h,b-h+i)\rightarrow (a+b+1,i-1)\oplus(a,h-1)\rightarrow(a,b)\rightarrow (a+h,b-h+i)[1]\]
which is what we need.

\end{proof}

The main result in this section is that the cluster map defined in Section \ref{Sec:cha} satisfies the cluster structure of $\C_n$.

\begin{thm}\label{Thm:2}
Given two basic maximal rigid objects $T'\oplus\overline{T}$ and
$T''\oplus\overline{T}$ in $\C_n$ such that both $T'$ and $T''$ are
indecomposable. Then $X_{T'}X_{T''}=X_E+X_{E'}$, where
$T'\rightarrow E\rightarrow T''\rightarrow T'[1]$ and
$T''\rightarrow E'\rightarrow T'\rightarrow T''[1]$ are the exchange
triangles.
\end{thm}

\begin{proof}

Since both the formulas of cluster maps in Theorem \ref{Thm:1} and the exchange triangles in Lemma \ref{Lem:mutofmr}
depend on the positions of objects, we can check that $X_{T'}X_{T''}=X_E+X_{E'}$ case by case.
We omit the details here.

\end{proof}

For the fixed maximal rigid object $T=\oplus_{i=1}^{n-1} T_i$, where
$T_i=(1,n-i)$, from Lemma \ref{Lem:conofmr} (or Proposition 3.4 in \cite{BMV}), we have that the matrix $A_T$
associated to $T$ is \[A_T=
\left(\begin{array}{cccccc}0&1&&&&\\
-2&0&1&&&\\
&-1&0&\ddots&&\\
&&\ddots&\ddots&\ddots&\\
&&&\ddots&0&1\\
&&&&-1&0\end{array}\right).\] Its Cartan part is of type
   $C_{n-1}$.

Let indr$\C_n$ be the set of indecomposable rigid objects in $\C_n$,
and $\chi_{A_{T}}$ be the set of cluster variables of the
skew-symmetrizable matrix $A_T$ associated to $T$. Then we have the
following:

\begin{cor} The map $X_{?}:
\mbox{indr}\C_n\longrightarrow \chi_{A_{T}}: M\mapsto X_M$ is a
bijection. This bijection induces a bijection between the set of isoclasses of basic
maximal rigid objects in $\C_n$ and the set of clusters of type $C_{n-1}$. Furthermore the algebra
generated by all $X_M$, where $M$ runs through indr$\C_n$, is
isomorphic to the cluster algebra of type $C_{n-1}$.
\end{cor}

\begin{proof} From Theorem \ref{Thm:1}, we have that $X_{T_i}=x_i$.
Then $X_{?}$ sends the couple $(\{T_1,\cdots, T_{n-1}\}, A_T)$ to
the initial seed $(\{x_1,\cdots, x_{n-1}\}, A_T)$ of the cluster
algebra $\A_{A_T}$. By Lemma \ref{Lem:mutofmr} that the exchange
graph of maximal rigid objects in $\C_n$ is connected. Then by
Theorem \ref{Thm:2} and Theorem 1.1 in \cite{BMV}, $X_{?}(\mbox{indr}\C_n)\subset\chi_{A_{T}}$, $X_{?}:
\mbox{indr}\C_n\longrightarrow \chi_{A_{T}}$ is surjective and $X_{?}$ induces a map
sending basic maximal rigid objects to clusters. It
follows from Theorem 3.7 that the denominators of $X_M$ are
different for all indecomposable rigid objects. Then $X_{?}:
\mbox{indr}\C_n\longrightarrow \chi_{A_{T}}$ is injective. Thus the statements hold.
\end{proof}

\section{Cluster complex of type C}\label{Sec:complex}

In this section, we use the results proved in Sections \ref{Sec:cha}, \ref{Sec:mu} to prove that
the combinatorics of indecomposable rigid objects in $\C_n$ encodes
the cluster combinatorics of the root system of type $C$ and type
$B$. Cluster complexes were defined in \cite{FZ2} for finite root systems.
They were realized by quiver representations via decorated representations \cite{MRZ},
and later via cluster categories of the corresponding quivers \cite{BMRRT,Z}.
Combining with the geometric description of the cluster complex of the root system of type $B$ \cite{FZ2},
Buan-Marsh-Vatne \cite{BMV} give a realization of this cluster complex via cluster tubes.

We recall the cluster complex associated to any finite root system
 from \cite{FZ2}. Let $\Phi$ be any finite root system with simple roots $\alpha_1,\cdots, \alpha_n$ and $\Phi_{\geq-1}$ the set of almost
positive roots in $\Phi$, i.e. the union of positive roots with
negative simple roots. Fomin and Zelevinsky define a function
\[(-\mid\mid-):\Phi_{\geq-1}\times\Phi_{\geq-1}\rightarrow\mathbb{Z}_{\geq0},\]
called the compatibility degree. A pair of roots $\alpha,\beta$ in
$\Phi_{\geq-1}$ are compatible if $(\alpha\mid\mid\beta)=0$. The
cluster complex $\Delta (\Phi)$ associated to $\Phi$ is a simplicial
complex, the set of vertices is $\Phi_{\geq-1}$ and the simplices
are mutually compatible subsets of $\Phi_{\geq-1}$. This
combinatorial object has many interesting properties and
applications, we refer to the survey \cite{FR} for further reading.

For the cluster tube $\C_n$, we call a set of indecomposable objects
a rigid subset provided the direct sum of all indecomposable objects
in this set is rigid. Now we define a simplicial complex associated
to the cluster tube $\C_n.$ We always assume $n>1$ throughout this
section.

\begin{defn} Let $\C_n$ be the cluster tube of rank $n$. The
cluster complex $\Delta (\C_n)$ associated to $\C_n$ is a simplicial complex whose
vertices are the isoclasses of indecomposable rigid objects and whose simplices are
the the isoclasses of rigid subsets of $\C_n$.
\end{defn}

Now fix a root system $\Phi^{C} $ of type $C_{n-1}$,  In \cite{FZ1},
Fomin-Zelevinsky give a bijection from the
set of cluster variables of the cluster algebras of type $C_{n-1}$
to $\Phi^{C}_{\geq-1}$ (they gave this bijection for all finite
root systems). Under this bijection, when a cluster variable $y$ is
expressed as
\[y=\frac{P(x)}{x^\alpha}\] where $P$ is a polynomial which is not
divisible by $x_i$ for every $i$, the corresponding almost
positive root is $\alpha$. Therefore, combining this bijection with
the bijective map $X_?$ from indr$\C_n$ to $\chi_{A_{T}}$ in
Section \ref{Sec:cha}, we have a bijection from $\Phi^{C}_{\geq-1}$ to indr$\C_n$.
This map is denoted by $M_T$. So, for any $\alpha\in \Phi^{C}_{\geq
-1}$, we denote by $M_T(\alpha)$ the object in $\C_n$ corresponding to $\alpha$ under
this bijection.

\begin{thm}\label{Thm:3}
Let $\Phi^{C}$ be the root system of type $\C_{n-1}$. Then the
 map $M_T$ induces an isomorphism from the cluster
complex $\Delta(\Phi ^{C})$ to the cluster complex $\Delta (\C_n)$,
which sends vertices to vertices, and simplices to simplices.
\end{thm}

To prove the theorem, we need some preparation:

\begin{defn}
For any two almost positive roots $\alpha,
\beta\in\Phi^{C}_{\geq-1}$, we define the $T-$compatibility degree
$(\alpha\mid\mid\beta)_{T}$ of $\alpha,\beta$ by
\[
(\alpha\mid\mid\beta)_{T}=\frac{\dim\Ext^1(M_T(\alpha),M_T(\beta))}{\dim\End(M_T(\alpha))}.
\]
\end{defn}

As in \cite{FZ2}, let $\sigma_i$ be the permutation of $\Phi^{C}_{\geq-1}$
defined as follows: \[\sigma_i(\alpha)=\begin{cases}\alpha & if
\alpha=-\alpha_j,j\neq i,\\s_i(\alpha) & otherwise,\end{cases}\]
where $s_i$ is the the Coxeter generator of the Weyl group of $\Phi^{C}$ corresponding to $i$.
We denote by $R$ the Coxeter element $\sigma_1\cdots\sigma_{n-1}$ in the Coxeter group (compare \cite{Z}).
The function $(-\mid\mid-)$ is determined by the following two properties \cite{FZ2}:
\[(-\alpha_i\mid\mid\beta)=\max([\beta:\alpha_i],0),\]
\[(R\alpha\mid\mid R\beta)=(\alpha\mid\mid\beta),.\]
where $[\beta:\alpha_i]$ denotes the coefficient of $\alpha_i$ in the expansion of $\beta$ in the basis $\{\alpha_1,\cdots,\alpha_n\}$.
We will prove that the function $(-\mid\mid-)_{T}$ satisfies the same properties.

\begin{lem}
For any $\alpha\in\Phi^{C}_{\geq-1}$, $M_T(R\alpha)=\tau M_T(\alpha)$.
\end{lem}

\begin{proof} Due to indr$\C_n=O\bigcup I\bigcup II$ (please see the
definitions of $O,I, II$ and Figure \ref{Fig:five} in section \ref{Sec:Pre}),
we divide the proof into several cases according to the position of
the object $M_T(\alpha)$.

1. When $M_T(\alpha)=(a,b)\in O$, then $\alpha=-\alpha_{n-b}$. We
have
\[R\alpha=\sigma_1\cdots\sigma_{n-1}(\alpha)=\sum_{i=1}^{n-b}\alpha_i.\]
Hence $M_T(R\alpha)=(n,n-b)=\tau(1,n-b)=\tau M_T(\alpha)$.

2. When $M_T(\alpha)=(a,b)\in I$, then $\alpha=\sum_{i=n-a-b+2}^{n-a+1}\alpha_i$. If $a\neq 2$, then
\[R\alpha=\sum_{i=n-a-b+3}^{n-a+2}\alpha_i.\] Hence
$M_T(R\alpha)=(a-1,b)=\tau (a,b)=\tau M_T(\alpha)$. If $a=2$, then
\[R\alpha=-\alpha_{n-b}.\] Hence $M_T(R\alpha)=(1,b)=\tau(2,b)=\tau
M_T(\alpha)$.

3. When $M_T(\alpha)=(2,n-1)\in II$, then
 $\alpha=\alpha_1+2\sum_{i=2}^{n-1}\alpha_i$. We have \[R\alpha=-\alpha_1.\] Hence $M_T(R\alpha)=(1,n-1)=\tau(2,n-1)=\tau
M_T(\alpha)$.

4. When $M_T(\alpha)=(a,n-1)\in II$ with $3\leq a\leq n$, then
 $\alpha=\alpha_1+2\sum_{i=2}^{n-a+1}\alpha_i$. We have \[R\alpha=\alpha_1+2\sum_{i=2}^{n-a+2}\alpha_i.\] Hence
$M_T(R\alpha)=(a-1,n-1)=\tau(a,n-1)=\tau M_T(\alpha)$.

5. When $M_T(\alpha)=(a,b)\in II$ with $a+b=n+1$ and $2\leq b\leq
n-2$, then
 $\alpha=\alpha_1+2\sum_{i=1}^{n-a+1}\alpha_i+\sum_{i=n-a+2}^{2n-a-b}\alpha_i$.
It follows that \[R\alpha=\sum_{i=2}^{n-a+2}\alpha_i.\] Hence
$M_T(R\alpha)=(a-1,b)=\tau(a,b)=\tau M_T(\alpha)$.

6. When $M_T(\alpha)=(n,1)\in II$, then
 $\alpha=\sum_{i=1}^{n-1}\alpha_i$. It follows that \[R\alpha=\alpha_2.\] Hence
$M_T(R\alpha)=(n-1,1)=\tau(n,1)=\tau M_T(\alpha)$.

7. When $M_T(\alpha)=(a,b)\in II$ with $a=n$ and $2\leq b\leq n-2$,
then  $\alpha=\sum_{i=1}^{n-b}\alpha_i$. It follows that
 \[R\alpha=\alpha_1+2\alpha_2+\sum_{i=3}^{n-b+1}\alpha_i.\] Hence
$M_T(R\alpha)=(a-1,b)=\tau(a,b)=\tau M_T(\alpha)$.

8. When $M_T(\alpha)=(a,b)\in II$ with $a < n, 1 < b < n -1$ and $a + b > n + 1$, then
$\alpha=\alpha_1+2\sum_{i=2}^{n-a+1}\alpha_i+\sum_{i=n-a+2}^{2n-a-b}\alpha_i$.
 It follows that
\[R\alpha=\alpha_1+2\sum_{i=2}^{n-a+2}\alpha_i+\sum_{i=n-a+3}^{2n-a-b+1}\alpha_i.\]
Hence $M_T(R\alpha)=(a-1,b)=\tau(a,b)=\tau M_T(\alpha)$.

The proof of this lemma is completed.
\end{proof}

Using the dimension formulas in the proof of Theorem \ref{Thm:1}, we have
the following fact:

\begin{lem}
For any positive root $\beta$ and any $i$,
$[\beta:\alpha_i]=\frac{\dim\Hom(T_i[-1],M_T(\beta))}{\dim\End(T_i)}$.
\end{lem}

\begin{lem} The $T-$compatibility degree satisfies the following conditions:
\begin{equation}
(-\alpha_i\mid\mid\beta)_{T}=\max([\beta:\alpha_i],0),
\end{equation}
\begin{equation}
(R\alpha\mid\mid R\beta)_{T}=(\alpha\mid\mid\beta)_{T},
\end{equation}
for any $\alpha,\beta\in\Phi^{C}_{\geq-1}$, any $1\leq i\leq n-1$.
\end{lem}

\begin{proof}
 By the definition,
\[(-\alpha_i\mid\mid\beta)_{T}=\frac{\dim\Ext^1(T_i,M_T(\beta))}{\dim\End(T_i)}
=\frac{\dim\Hom(T_i[-1],M_T(\beta))}{\dim\End(T_i)}.\] Then by Lemma
5.3 it equals $[\beta:\alpha_i]$ if $\beta$ is a positive root, or 0
otherwise. This proves that (1) holds. From Lemma 5.2, we have
\[(R\alpha\mid\mid
R\beta)_T=\frac{\dim\Ext^1_{\C}(M_T(R\alpha),M_T(R\beta))}{\dim\End_{\C}(M_T(R\alpha))}
=\frac{\dim\Ext^1_{\C}(\tau M_T(\alpha),\tau
M_T(\beta))}{\dim\End_{\C}(\tau M_T(\alpha))}
=(\alpha\mid\mid\beta)_T.\]
This proves that (2) holds.

\end{proof}

\begin{proof}[of Theorem~{\rm\ref{Thm:3}}]
By Lemmas 5.3, 5.4, the compatibility degree $(-||-)_T$
 is the same as $(-||-)$ in \cite{FZ2}. It follows that
$\alpha,\beta$ are compatible if and only if
$M_T(\alpha),M_T(\beta)$ form a rigid subset. Therefore $M_T$
induces the desired bijection from $\Delta(\Phi^{C})$ to
$\Delta(\C_n)$.

\end{proof}

Let $\Phi^{B}$ be the root system of type $B_{n-1}$. Then
$\Phi^B_{\geq-1}$ is the dual of $\Phi^C_{\geq-1}$ via
$\alpha\mapsto \alpha^{\vee}$. So
$(\alpha\mid\mid\beta)=(\beta^\vee\mid\mid\alpha^\vee)$ (Proposition
3.15 \cite{FZ2}).  Then we have the following corollary (compare with Theorem
3.5 in \cite{BMV}).

\begin{cor} Let $\Phi^{B}$ be the root system of type $B_{n-1}$. Then the cluster complex $\Delta(\Phi^{B})$
of $\Phi$ of type $B_{n-1}$ is isomorphic to
$\Delta(\C_n)$.
\end{cor}

\begin{acknowledgements}\label{ackref}
The authors would like to thank Aslak Bakke Buan, Bin Li, Yann Palu and Dong Yang for their
interesting, helpful discussions and suggestions. The authors would like to thank the referee for his/her helpful comments to improve the paper.
\end{acknowledgements}

\affiliationone{
   Yu Zhou\\
   Department of Mathematical Sciences, Tsinghua University, 100084, Beijing\\
   China
   \email{yuzhoumath@gmail.com}}
\affiliationtwo{
   Bin Zhu\\
   Department of Mathematical Sciences, Tsinghua University, 100084, Beijing\\
   China
   \email{bzhu@math.tsinghua.edu.cn}}
\affiliationthree{%
   Current address:\\
   Fakult\"{a}t f\"{u}r Mathematik, Universit\"{a}t Bielefeld, D-33501, Bielefeld\\
   Germany
   \email{yuzhoumath@gmail.com}}

\begin{thebibliography}{9}
%



%
\bibitem{Am}
 {\bibname C. Amiot},
 `Cluster categories for algebras of global dimension $2$ and quivers with potential',
 {\em Annales de l'institut Fourier }59 (2009) 2525--2590.
%
\bibitem{BKL}
 {\bibname M. Barot, D. Kussin \and H. Lenzing},
 `The Grothendieck group of a cluster category',
 {\em Journal of Pure and Applied Algebra }212 (2008) 33--46.
%
\bibitem{BIRS}
 {\bibname A. B. Buan, O. Iyama, I. Reiten \and J. Scott},
 `Cluster structures for $2-$Calabi-Yau categories and unipotent groups',
 {\em Compositio Mathematica }145 (2009) 1035--1079.
%
\bibitem{BMRRT}
 {\bibname A. B. Buan, R. Marsh, M. Reineke, I. Reiten \and G. Todorov},
 `Tilting theory and cluster combinatorics',
 {\em Advances in mathematics }204 (2006) 572--618.
%
\bibitem{BMV}
 {\bibname A. B. Buan, R. Marsh \and D. Vatne},
 `Cluster structures from $2-$Calabi-Yau categories with loops',
 {\em Mathematische Zeitschrift }265 (2010) 951--970.
%
\bibitem{BIKR}
 {\bibname I. Burban, O. Iyama, B. Keller \and I. Reiten},
 `Cluster tilting for one-dimensional hypersurface singularities',
 {\em Advances in Mathematics }217 (2008) 2443--2484.
%
\bibitem{CC}
 {\bibname P. Caldero \and F. Chapoton},
 `Cluster algebras as Hall algebras of quiver representations',
 {\em Commentarii Mathematici Helvetici }81 (2006) 595--616.
%
\bibitem{CCS}
 {\bibname P. Caldero, F. Chapoton \and R. Schiffler},
 `Quivers with relations arising from clusters ($A_n$ case)',
 {\em Transactions of the American Mathematical Society }358 (2006) 1347--1364.
%
\bibitem{CK}
 {\bibname P. Caldero \and B. Keller},
 `From triangulated categories to cluster algebras II',
 {\em Annales Scientifiques de l'Ecole Normale Sup\'{e}rieure }39 (2006) 983--1009.
%
\bibitem{Du1}
 {\bibname G.Dupont},
 `An approach to non-simply laced cluster algebras',
 {\em Journal of Algebra }320 (2008) 1626--1661.
%
\bibitem{Du2}
 {\bibname G.Dupont},
 `Cluster multiplication in regular components via generalized Chebyshev polynomials',
 {\em Algebras and Representation Theory }15 (2012) 527--549.
%
\bibitem{DK}
 {\bibname R. Dehy \and B. Keller},
 `On the Combinatorics of Rigid Objects in 2-Calabi-Yau Categories',
 {\em International Mathematics Research Notices }2008 (2008) doi:10.1093/imrn/rnn029.
%
\bibitem{DX}
 {\bibname M. Ding \and F. Xu},
 `Cluster characters for cyclic quivers',
 {\em Frontiers of Mathematics in China }7 (2012) 679--693.
%
\bibitem{F}
 {\bibname S. Fomin},
 `Total positivity and cluster algebras',
 {\em Proceedings of the International Congress of Mathematicians }901 (2010) 125--145.
%
\bibitem{FR}
 {\bibname S. Fomin \and N. Reading},
 `Root systems and generalized associahedra',
 {\em IAS/Park City Mathematics Series }14 (2004).
%
\bibitem{FZ1}
 {\bibname S. Fomin \and A. Zelevinsky},
 `Cluster Algebras I: Foundations',
 {\em Journal of the American Mathematical Society }15 (2002) 497--529.
%
\bibitem{FZ2}
 {\bibname S. Fomin \and A. Zelevinsky},
 `Y-systems and generalized associahedra',
 {\em Annals of Mathematics }158 (2003) 977--1018.
%
\bibitem{FZ3}
 {\bibname S. Fomin \and A. Zelevinsky},
 `Cluster algebras II: Finite type classification',
 {\em Inventiones Mathematicae }154 (2003) 63--121.
%
\bibitem{FZ4}
 {\bibname S. Fomin \and A. Zelevinsky},
 `Cluster algebras: Notes for the CDM-03 Conference',
 {\em Current Developments in Mathematics }2003 (2003) 1--34.
%
\bibitem{FK}
 {\bibname C. Fu \and B. Keller},
 `On cluster algebras with coefficients and $2-$Calabi-Yau categories',
 {\em Transactions of the American Mathematical Society }362 (2010) 859--895.
%
\bibitem{GLS1}
 {\bibname C. Gei\ss, B. Leclerc \and J. Shr\"{o}er},
 `Rigid modules over preprojective algebras',
 {\em Inventiones Mathematicae }165 (2006) 589--632.
%
\bibitem{GLS2}
 {\bibname C. Gei\ss, B. Leclerc \and J. Shr\"{o}er},
 `Preprojective algebras and cluster algebras',
 {\em Trends in Representation Theory of Algebras and Related Topics, EMS Ser. Congr. Rep., Eur. Math. Soc., Z\"{u}rich }(2008) 253--283.
%
\bibitem{HR}
 {\bibname D. Happel \and C. M. Ringel},
 `Construction of tilted algebras',
 {\em Representations of Algebras } (1981) 125--144.
%
\bibitem{Hu}
 {\bibname J. E. Humphreys},
 `Introduction to Lie Algebras and Representation Theory',
 Springer-Verlag, New York, Heidelberg, Berlin (1972).
%
\bibitem{I}
 {\bibname O. Iyama},
 `Higher dimensional Auslander-Reiten theory on maximal orthogonal subcategories',
 {\em Advances in Mathematics }210 (2007) 22--50.
%
\bibitem{IY}
 {\bibname O. Iyama \and Y. Yoshino},
 `Mutations in triangulated categories and rigid Cohen-Macaulay modules',
 {\em Inventiones Mathematicae }172 (2008) 117¨C-168.
%
\bibitem{Ke1}
 {\bibname B. Keller},
 `On triangulated orbit categories',
 {\em Documenta Math }10 (2005) 551--581.
%
\bibitem{Ke2}
 {\bibname B. Keller},
 `Cluster algebras, quiver representations and triangulated categories',
  in: Triangulated categories, 76--160, London Math. Soc. Lecture Note Ser. 375, Cambridge Univ. Press, Cambridge, 2010.
%
\bibitem{Ke3}
 {\bibname B. Keller},
 `Triangulated Calabi-Yau categories',
 {\em Trends in representation theory of algebras and related topics } (2008) 467--489.
%
\bibitem{KR}
 {\bibname B. Keller \and I. Reiten},
 `Cluster-tilted algebras are Gorenstein and stably Calabi-Yau',
 {\em Advances in Mathematics }211 (2007) 123--151.
%
\bibitem{KZ}
 {\bibname S. Koenig \and B. Zhu},
 `From triangulated categories to abelian categories: cluster tilting in a  general framework',
 {\em Mathematische Zeitschrift }258 (2008) 143--160.
%
\bibitem{MRZ}
 {\bibname R. Marsh, M. Reineke \and A. Zelevinsky},
 `Generalized associahedra via quiver representations',
 {\em Transactions of the American Mathematical Society }355 (2003) 4171--4186.
%
\bibitem{Pa1}
 {\bibname Y. Palu},
 `Cluster characters for $2-$Calabi-Yau triangulated categories',
 {\em Annales de l'institut Fourier }58 (2008) 2221--2248.
%
\bibitem{Pa2}
 {\bibname Y. Palu},
 `Grothendieck group and generalized mutation rule for $2-$Calabi-Yau triangulated categories',
 {\em Journal of Pure and Applied Algebra }213 (2009) 1438--1449.
%
\bibitem{Pla1}
 {\bibname P. Plamondon},
 `Cluster characters for cluster categories with infinite-dimensional morphism spaces',
 {\em Advances in Mathematics }227 (2011) 1--39.
%
\bibitem{Pla2}
 {\bibname P. Plamondon},
 `Cluster algebras via cluster categories with infinite-dimensional morphism spaces',
 {\em Compositio Mathematica }147 (2011) 1921--1954.
%
\bibitem{Re}
 {\bibname I. Reiten},
 `Cluster categories',
 {\em Proceedings of the International Congress of Mathematicians, Hyderabad, India }1 (2010) 558--594.
%
\bibitem{Rin}
 {\bibname C. M. Ringel},
 `Some Remarks Concerning Tilting Modules and Tilted Algebras. Origin. Relevance. Future. (An appendix to the Handbook of Tilting Theory)',
 Edited by Lidia Angeleri-H\"ugel, Dieter Happel and Henning Krause, Cambridge University Press (2007), LMS Lecture Notes, Series 332.
%
\bibitem{V}
 {\bibname D. Vatne},
 `Endomorphism rings of maximal rigid objects in cluster tubes',
 {\em Colloqium Mathematicum }123 (2011) 63--93.
%
\bibitem{Y}
 {\bibname D. Yang},
 `Endomorphism algebras of maximal rigid objects in cluster tubes',
 {\em Comm. Algebra }40 (2012) 4347--4371.
%
\bibitem{ZZ}
 {\bibname Y. Zhou \and B. Zhu},
 `Maximal rigid subcategories in $2-$Calabi-Yau triangulated categories',
 {\em Journal of Algebra }348 (2011) 49--60.
%
\bibitem{Z}
 {\bibname B. Zhu},
 `BGP-reflection functors and cluster combinatorics',
 {\em Journal of Pure and Applied Algebra }209 (2007) 497--506.
\end{thebibliography}
\end{document}